\documentclass[12pt,leqno]{article}

\usepackage{amsmath,amssymb,amsthm}
\usepackage[all]{xy}
\usepackage{lscape}

\makeatletter

\newtheorem{theorem}{Theorem}

\newtheorem{lemma}[theorem]{Lemma}
\newtheorem{corollary}[theorem]{Corollary}

\theoremstyle{definition}

\newtheorem{example}[theorem]{Example}
\newtheorem{definition}[theorem]{Definition}

\theoremstyle{remark}

\theoremstyle{remark}
\newtheorem{remark}[theorem]{Remark}

\def\({{\rm (}}
\def\){{\rm )}}

\let\Mathrm\operator@font

\let\Bbb\mathbb
\newcommand{\fm}{\ensuremath{\mathfrak m}}

\def\standop#1{\mathop{\Mathrm #1}\nolimits}
\def\difstop#1#2{\expandafter\def\csname #1\endcsname{\standop{#2}}}
\def\defstop#1{\difstop{#1}{#1}}

\defstop{AB}
\defstop{ann}
\defstop{Ass}

\defstop{CMFI}
\defstop{codim}
\defstop{Coh}
\defstop{coht}
\defstop{Coker}
\defstop{Cone}
\defstop{Cl}
\defstop{Cox}

\defstop{depth}

\defstop{EM}
\defstop{embdim}
\defstop{End}
\defstop{ev}
\defstop{Ext}

\defstop{Flat}
\defstop{Func}

\defstop{GD}
\defstop{Good}

\difstop{height}{ht}
\defstop{Hom}

\difstop{Image}{Im}
\defstop{ind}

\defstop{Ker}

\defstop{Lch}
\defstop{length}
\defstop{Lin}
\defstop{Lqc}
\defstop{lqc}

\defstop{Mat}
\defstop{Max}
\defstop{Min}
\defstop{Mod}
\defstop{Mor}
\defstop{MCM}

\defstop{Nerve}
\defstop{NonSFR}
\defstop{NonCMFI}
\defstop{NonNor}
\defstop{Nor}

\defstop{PA}
\defstop{PM}
\defstop{Proj}
\defstop{Pic}

\defstop{Qch}
\defstop{qch}

\defstop{rad}
\defstop{rank}
\def\red{_{\Mathrm{red}}}
\defstop{res}
\defstop{Reg}

\defstop{Spec}
\defstop{supp}
\defstop{Supp}
\defstop{Sym}
\defstop{Sing}
\defstop{SFR}
\defstop{Soc}

\difstop{tdeg}{trans.deg}

\defstop{Tor}
\difstop{trace}{tr}

\defstop{Zar}

\def\fm{\mathfrak{m}}

\def\sdarrow#1{\downarrow\hbox to 0pt{\scriptsize$#1$\hss}}
\def\suarrow#1{\uparrow\hbox to 0pt{\scriptsize$#1$\hss}}
\def\ssearrow#1{\searrow\hbox to 0pt{\scriptsize$#1$\hss}}

\def\section{\@startsection{section}{1}{\z@ }%
{-3.5ex plus -1ex minus -.2ex}{2.3ex plus .2ex}{\bf }}

\long\def\refname{\par\kern -3ex
\begin{center}\rm R\sc{eferences}\end{center}\par\kern 
-2ex}

\def\@seccntformat#1{\csname the#1\endcsname.\quad}

\def\@@@sect#1#2#3#4#5#6[#7]#8{%
   \ifnum #2>\c@secnumdepth 
      \def \@svsec {}\else \refstepcounter {#1}%
      \def\@svsec{}
   \fi 
   \@tempskipa #5\relax 
   \ifdim \@tempskipa >\z@ 
     \begingroup #6\relax \@hangfrom {\hskip #3\relax 
     \@svsec}{\interlinepenalty \@M #8\par }\endgroup 
     \csname #1mark\endcsname {#7}
   \else 
   \def \@svsechd {#6\hskip #3\@svsec #8\csname #1mark\endcsname {#7}}
   \fi \@xsect {#5}}

\def\@@@startsection#1#2#3#4#5#6{%
 \if@noskipsec \leavevmode \fi \par \@tempskipa #4\relax \@afterindenttrue 
 \ifdim \@tempskipa <\z@ \@tempskipa -\@tempskipa \@afterindentfalse 
 \fi \if@nobreak \everypar {}\else \addpenalty {\@secpenalty }\addvspace 
  {\@tempskipa }\fi \@ifstar {\@ssect {#3}{#4}{#5}{#6}}{\@dblarg 
  {\@@@sect {#1}{#2}{#3}{#4}{#5}{#6}}}}

\def\theparagraph{\thesection.\arabic{paragraph}}
\def\aparagraph{\@@@startsection{paragraph}{2}{\z@ }%
              {1.75ex plus .2ex minus .15ex}{-1em}{\bf(\theparagraph) } }
\def\paragraph{\@@@startsection{paragraph}{2}{\z@ }%
              {1.75ex plus .2ex minus .15ex}{-1em}{}{\bf(\theparagraph)} }

\let\c@theorem\c@paragraph

\title{$F$-finiteness of homomorphisms and its descent}
\author{M{\sc itsuyasu} H{\sc ashimoto}}
\date{\normalsize
Graduate School of Mathematics, Nagoya University\\
Chikusa-ku,  Nagoya 464--8602 JAPAN\\
{\small \tt hasimoto@math.nagoya-u.ac.jp}}

\begin{document}

\maketitle

\begin{abstract}
Let $p$ be a prime number.
We define the notion of $F$-finiteness of homomorphisms of 
$\Bbb F_p$-algebras, and
discuss some basic properties.
In particular, we prove a sort of descent theorem on $F$-finiteness
of homomorphisms of $\Bbb F_p$-algebras.
As a corollary, we prove the following.
Let $g:B\rightarrow C$ be a homomorphism of Noetherian $\Bbb F_p$-algebras.
If $g$ is faithfully flat reduced and $C$ is $F$-finite, 
then $B$ is $F$-finite.
This is a generalization of Seydi's result on excellent local rings of
characteristic $p$.
\end{abstract}

\section{Introduction}
Throughout this paper, $p$ denotes a prime number, and $\Bbb F_p$ denotes
the finite field with $p$ elements.
In commutative algebra of characteristic $p$, $F$-finiteness of rings
are commonly used for a general assumption which guarantees the \lq\lq 
tameness'' 
of the theory, as well as excellence.
Although $F$-finiteness for a Noetherian $\Bbb F_p$-algebra is stronger
than excellence \cite{Kunz},
$F$-finiteness is not so restrictive for practical use.
A perfect field is $F$-finite.
An algebra essentially of finite type over an $F$-finite ring is $F$-finite.
An ideal-adic completion of a Noetherian $F$-finite ring is again $F$-finite.
See Example~\ref{example.ex} and Example~\ref{complete2.ex}.
In this paper, replacing the absolute Frobenius map by the relative one,
we define the $F$-finiteness of homomorphism between rings of characteristic
$p$.
We say that an $\Bbb F_p$-algebra map $A\rightarrow B$ is $F$-finite
(or $B$ is $F$-finite over $A$) if the relative Frobenius map 
(Radu--Andr\'e homomorphism)
$\Phi_1(A,B):B^{(1)}\otimes_{A^{(1)}}A\rightarrow B$ is finite
(Definition~\ref{main.def}, see section~2 for the notation).
Thus a ring $B$ of characteristic $p$ is $F$-finite if and only
if it is $F$-finite over $\Bbb F_p$.
Replacing absolute Frobenius by relative Frobenius,
we get definitions and results on homomorphisms instead of rings.
This is a common idea in
\cite{Radu}, \cite{Andre}, 
\cite{Andre2}, \cite{Dumitrescu2}, \cite{Dumitrescu}, 
\cite{Enescu}, \cite{Hashimoto}, and \cite{Hashimoto2}.

In section~2,
we discuss basic properties of $F$-finiteness of homomorphisms and rings.
Some of well-known properties
of $F$-finiteness of rings are naturally generalized to those for
$F$-finiteness of homomorphisms.
$F$-finiteness of homomorphisms has connections with that for rings.
For example, if $A\rightarrow B$ is $F$-finite and $A$ is $F$-finite,
then $B$ is $F$-finite (Lemma~\ref{basic.thm}).

In section~3, we prove the main theorem (Theorem~\ref{main.thm}).
This is a sort of descent of $F$-finiteness.
As a corollary, we prove that for a faithfully flat reduced homomorphism
of Noetherian rings $g:B\rightarrow C$, if $C$ is $F$-finite, then $B$
is $F$-finite.
Considering the case that $f$ is a completion of a Noetherian
local ring, we recover
Seydi's result on excellent local rings of characteristic $p$ \cite{Seydi}.

\medskip
Acknowledgement: The author is grateful to 
Professor K.~Fujiwara,
Professor J.-i.~Nishimura,
Professor C.~Rotthaus,
Dr. A. Sannai,
Professor S.~Takagi,
and
Professor H.~Tanimoto
for valuable advice.

\section{$F$-finiteness of homomorphisms}
Let $k$ be a perfect field of characteristic $p$, 
and $r\in\Bbb Z$.
For a $k$-space $V$, the additive group $V$ with the new $k$-space structure
$\alpha\cdot v=\alpha^{p^{-r}}v$ is denoted by $V^{(r)}$.
An element $v$ of $V$, viewed as an element of $V^{(r)}$ is (sometimes) denoted
by $v^{(r)}$.
If $A$ is a $k$-algebra, then $A^{(r)}$ is a $k$-algebra with the product
$a^{(r)}\cdot b^{(r)}=(ab)^{(r)}$.
We denote the Frobenius map $A\rightarrow A$ $(a\mapsto a^p)$ by $F$ or $F_A$.
Note that $F^e:A^{(r+e)}\rightarrow A^{(r)}$ is a $k$-algebra map.
Throughout the article, we regard $A^{(r)}$ as an $A^{(r+e)}$-algebra 
through $F^e$ ($A$ is viewed as $A^{(0)}$).
For an $A$-module $M$, the action $a^{(r)}\cdot m^{(r)}=(am)^{(r)}$ makes
$M^{(r)}$ an $A^{(r)}$ module.
If $I$ is an ideal of $A$, then $I^{(r)}$ is an ideal of $A^{(r)}$.
If $e\geq 0$, then $I^{(e)}A=I^{[p^e]}$, where $I^{[p^e]}$ is the ideal
of $A$ generated by $\{a^{p^e}\mid a\in I\}$.
In commutative algebra, $A^{(r)}$ is also denoted by ${}^{-r}A$.
We employ the notation more consistent with that in representation theory 
--- the $e$th Frobenius twist of $V$ 
is denoted by $V^{(e)}$, see \cite{Jantzen}.
We use this notation for $k=\Bbb F_p$.

Let $A\rightarrow B$ be an $\Bbb F_p$-algebra map, and $e\geq 0$.
Then the relative Frobenius map (or Radu--Andr\'e homomorphism) 
$\Phi_e(A,B):B^{(e)}\otimes_{A^{(e)}}A\rightarrow B$
is defined by $\Phi_e(A,B)(b^{(e)}\otimes a)=b^{p^e}a$.

\begin{definition}\label{main.def}
An $\Bbb F_p$-algebra map $A\rightarrow B$ is said to be {\em $F$-finite}
if $\Phi_1(A,B): B^{(1)}\otimes_{A^{(1)}}A\rightarrow B$ is finite.
That is, $B$ is a finitely generated $B^{(1)}\otimes_{A^{(1)}}A$-module
through $\Phi_1(A,B)$.
We also say that $B$ is $F$-finite over $A$.
\end{definition}

\begin{lemma}\label{basic.thm}
Let $f:A\rightarrow B$, $g:B\rightarrow C$, and $h:A\rightarrow \tilde A$
be $\Bbb F_p$-algebra maps, and $\tilde B:=\tilde A\otimes_A B$.
\begin{enumerate}
\item[\bf 1]
The following are equivalent.
\begin{enumerate}
\item[\bf a] $f$ is $F$-finite.
That is, $\Phi_1(A,B)$ is finite.
\item[\bf b] For any $e>0$, $\Phi_e(A,B)$ is finite.
\item[\bf c] For some $e>0$, $\Phi_e(A,B)$ is finite.
\end{enumerate}
\item[\bf 2] If $f$ and $g$ are $F$-finite, then so is $gf$.
\item[\bf 3] If $gf$ is $F$-finite, then so is $g$.
\item[\bf 4] The ring $A$ is $F$-finite \(that is, the Frobenius map
$F_A:A^{(1)}\rightarrow A$ is finite\) 
if and only if the unique homomorphism $\Bbb F_p\rightarrow A$ 
is $F$-finite.
\item[\bf 5] If $f:A\rightarrow B$ is $F$-finite, then the base change
$\tilde f:\tilde A\rightarrow\tilde B$ is $F$-finite.
\item[\bf 6] If $B$ is $F$-finite, then $f$ is $F$-finite.
\item[\bf 7] If $A$ and $f$ are $F$-finite, then $B$ is $F$-finite.
\end{enumerate}
\end{lemma}

\begin{proof}
{\bf 1} This is immediate, using \cite[Lemma~4.1, {\bf 2}]{Hashimoto}.
{\bf 2} and {\bf 3} follow from \cite[Lemma~4.1, {\bf 1}]{Hashimoto}.
{\bf 4} follows from \cite[Lemma~4.1, {\bf 5}]{Hashimoto}.
{\bf 5} follows from \cite[Lemma~4.1, {\bf 4}]{Hashimoto}.
{\bf 6} follows from {\bf 3} and {\bf 4}.
{\bf 7} follows from {\bf 2} and {\bf 4}.
\end{proof}

\begin{example}\label{example.ex}
Let $e\geq 1$, and
$f:A\rightarrow B$ be an $\Bbb F_p$-algebra map.
\begin{enumerate}
\item[\bf 1] If $B=A[x]$ is a polynomial ring, then it is 
$F$-finite over $A$.
\item[\bf 2] If $B=A_S$ is a localization of $A$ by a multiplicatively
closed subset $S$ of $A$, then $\Phi_e(A,B)$ is an isomorphism.
In particular, $B$ is $F$-finite over $A$.
\item[\bf 3] If $B=A/I$ with $I$ an ideal of $A$, then 
\[
B^{(e)}\otimes_{A^{(e)}}A
\cong
(A^{(e)}/I^{(e)})\otimes_{A^{(e)}}A
\cong
A/I^{(e)}A=A/I^{[p^e]}.
\]
Under this identification, $\Phi_e(A,B)$ is identified with the
projection $A/I^{[p^e]}\rightarrow A/I$.
In particular, $B$ is $F$-finite over $A$.
\item[\bf 4] If $B$ is essentially of finite type over $A$, then
$B$ is $F$-finite over $A$.
\end{enumerate}
\end{example}

\begin{proof}
{\bf 1} The image of $\Phi_1(A,B)$ is $A[x^p]$, and hence $B$ is
generated by $1,x,\ldots,x^{p-1}$ over it.
{\bf 2} Note that $B^{(e)}$ is identified with $(A^{(e)})_{S^{(e)}}$,
where $S^{(e)}=\{s^{(e)}\mid s\in S\}$.
So $B^{(e)}\otimes_{A^{(e)}}A$ is identified with
$(A^{(e)})_{S^{(e)}}\otimes_{A^{(e)}}A\cong A_{S^{(e)}}$,
and $\Phi_e(A,B)$ is identified with the isomorphism $A_{S^{(e)}}\cong
A_S$.
{\bf 3} is obvious.
{\bf 4} This is a consequence of {\bf 1}, {\bf 2}, {\bf 3}, and
Lemma~\ref{basic.thm}, {\bf 2}.
\end{proof}

\begin{lemma}
Let $A\xrightarrow f B\xrightarrow g C$ be a sequence of $\Bbb F_p$-algebra
maps.
Then for $e>0$, the diagram
\[
\xymatrix{
B^{(e)}\otimes_{A^{(e)}}A \ar[r]^-{\Phi_e(A,B)} \ar[d]^{g^{(e)}\otimes 1}
&
B \ar[d]^g \\
C^{(e)}\otimes_{A^{(e)}}A \ar[r]^-{\Phi_e(A,C)} &
C
}
\]
is commutative.
\end{lemma}

\begin{proof}
This is straightforward.
\end{proof}

\begin{lemma}
\label{finite.thm}
Let $A\xrightarrow f B\xrightarrow g C$ be a sequence of 
$\Bbb F_p$-algebra maps, and assume that $C$ is $F$-finite over $A$.
If $g$ is finite and injective, and $B^{(e)}\otimes_{A^{(e)}}A$ is 
Noetherian for some $e>0$, then $B$ is $F$-finite over $A$.
\end{lemma}

\begin{proof}
By assumption, $C^{(e)}\otimes_{A^{(e)}}A$ is finite over
$B^{(e)}\otimes_{A^{(e)}}A$, and $C$ is finite over $C^{(e)}\otimes
_{A^{(e)}}A$.
So $C$ is finite over $B^{(e)}\otimes_{A^{(e)}}A$.
As $B$ is a $B^{(e)}\otimes_{A^{(e)}}A$-submodule of $C$
and $B^{(e)}\otimes_{A^{(e)}}A$ is Noetherian, 
$B$ is finite over $B^{(e)}\otimes_{A^{(e)}}A$.
\end{proof}

\begin{lemma}\label{matsumura.thm}
Let $A\rightarrow B$ be a ring homomorphism, and $I$ a
finitely generated nilpotent ideal of $B$.
If $B/I$ is $A$-finite, then $B$ is $A$-finite.
\end{lemma}

\begin{proof}
As $I^i/I^{i+1}$ is $B/I$-finite for each $i$, it is also $A$-finite.
So $B/I^r$ is $A$-finite for each $r$.
Taking $r$ large, $B$ is $A$-finite.
\end{proof}

\begin{lemma}\label{nilpt.thm}
Let $f:A\rightarrow B$ be an $\Bbb F_p$-algebra map,
and $I$ a finitely generated nilpotent ideal of $B$.
If $B/I$ is $F$-finite over $A$, then $B$ is $F$-finite over $A$.
\end{lemma}

\begin{proof}
As $B/I$ is $F$-finite over $A$, $B/I$ is
$(B^{(1)}/I^{(1)})\otimes_{A^{(1)}}A$-finite.
So $B/I$ is also $B^{(1)}\otimes_{A^{(1)}}A$-finite.
By Lemma~\ref{matsumura.thm}, $B$ is 
$B^{(1)}\otimes_{A^{(1)}}A$-finite.
\end{proof}

For the absolute $F$-finiteness, we have a better result.

\begin{lemma}\label{complete.thm}
Let $B$ be an $\Bbb F_p$-algebra, and $I$ a finitely generated ideal of $B$.
If $B$ is $I$-adically complete and $B/I$ is $F$-finite, then $B$ is 
$F$-finite.
\end{lemma}

\begin{proof}
$B/I$ is $B^{(1)}/I^{(1)}$-finite.
So $B/I^{(1)}B$ is $B^{(1)}$-finite by Lemma~\ref{matsumura.thm}.
As $\bigcap_i I^i=0$, we have $\bigcap_i (I^{(1)})^iB=0$.
Moreover, $B^{(1)}$ is $I^{(1)}$-adically complete.
Hence $B$ is $B^{(1)}$-finite by \cite[Theorem~8.4]{Matsumura}.
\end{proof}

\begin{example}\label{complete2.ex}
Let $A$ be an $\Bbb F_p$-algebra.
\begin{enumerate}
\item[\bf 1] If $A$ is $F$-finite, then the formal power series ring
$A[[x]]$ is so.
\item[\bf 2] Let $J$ be an ideal of $A$.
If $A$ is Noetherian and $A/J$ is $F$-finite, then
the $J$-adic completion $A^*$ of $A$ is $F$-finite.
\item[\bf 3] If $(A,\frak m)$ is complete local and $A/\frak m$ is
$F$-finite, then $A$ is $F$-finite.
\end{enumerate}
\end{example}

\begin{proof}
For each of {\bf 1--3}, we use Lemma~\ref{complete.thm}.
{\bf 1} Set $B=A[[x]]$ and $I=Bx$.
Then $B/I\cong A$ is $F$-finite.
{\bf 2} Set $B=A^*$ and $I=JB$.
Then $B/I\cong A/J$ is $F$-finite.
{\bf 3} is immediate.
\end{proof}

Let $A$ be a Noetherian ring and $I$ its ideal.
If $A$ is $I$-adically complete and $A/I$ is Nagata,
then $A$ is Nagata \cite{Marot}.
If $A$ is semi-local, $I$-adically complete, and $A/I$ is
quasi-excellent, then $A$ is quasi-excellent \cite{Rotthaus}.
See also \cite{Nishimura}.

\begin{lemma}\label{trivial.thm}
Let $A$ be an $\Bbb F_p$-algebra, and $B$ and $C$ be $A$-algebras.
If $B$ and $C$ are $F$-finite over $A$, then 
\begin{enumerate}
\item[\bf 1] $B\otimes_A C$ is $F$-finite over $A$.
\item[\bf 2] $B\times C$ is $F$-finite over $A$.
\end{enumerate}
\end{lemma}

\begin{proof}
{\bf 1}
$B$ is $F$-finite over $A$, and $B\otimes_A C$ is $F$-finite over $B$
by Lemma~\ref{basic.thm}, {\bf 5}.
By Lemma~\ref{basic.thm}, {\bf 2}, $B\otimes_A C$ is $F$-finite over $A$.

{\bf 2} Both $B$ and $C$ are finite over 
$(B\times C)^{(1)}\otimes_{A^{(1)}}A$, and so is $B\times C$.
\end{proof}

\begin{lemma}
Let $A\rightarrow B$ be an $\Bbb F_p$-algebra map, and
assume that $B$ and $B^{(e)}\otimes_{A^{(e)}}A$ are Noetherian for some $e>0$.
Then $B$ is $F$-finite over $A$ if and only if $B/P$ is $F$-finite
over $A$ for every minimal prime $P$ of $B$.
\end{lemma}

\begin{proof}
The \lq only if' part is obvious by Example~\ref{example.ex}, {\bf 3}.
We prove the converse.
Let $\Min B$ be the set of minimal primes of $B$.
Then $\prod_{P\in\Min B} B/P$ is $F$-finite over $A$ by
Lemma~\ref{trivial.thm}.
As $B\red\rightarrow\prod_{P\in\Min B}B/P$ is finite injective, and
$B\red^{(e)}\otimes_{A^{(e)}}A$ is Noetherian, 
$B\red$ is $F$-finite over $A$ by Lemma~\ref{finite.thm}.
As $B$ is Noetherian, $B$ is $F$-finite over $A$ by 
Lemma~\ref{nilpt.thm}.
\end{proof}

\begin{remark}
Fogarty asserted that an $\Bbb F_p$-algebra
map $A\rightarrow B$ with $B$ Noetherian 
is $F$-finite if and only if
the module of K\"ahler differentials $\Omega_{B/A}$ is
a finite $B$-module \cite[Proposition~1]{Fogarty}.
The \lq only if' part is true and easy.
The proof of \lq if' part therein has a gap.
Although $R_1$ in step~(iii) is assumed to be Noetherian,
it is not proved that $R'$ in step~(iv) is Noetherian.
The author does not know if this direction is true
or not.
\end{remark}

\section{Descent of $F$-finiteness}

In this section, we prove a sort of descent theorem on $F$-finiteness
of homomorphisms.

\begin{lemma}
Let $R$ be a commutative ring, $\varphi:M\rightarrow N$ and
$h:F\rightarrow G$ be $R$-linear maps.
If $\varphi$ is $R$-pure and 
$1_N\otimes h:N\otimes F\rightarrow N\otimes G$ is surjective,
then $1_M\otimes h:M\otimes F\rightarrow M\otimes G$ is surjective.
\end{lemma}

\begin{proof}
Let $C:=\Coker h$.
Then by assumption, $N\otimes C=0$.
By the injectivity of $\varphi\otimes 1_C:M\otimes C\rightarrow N\otimes C$,
we have that $M\otimes C=0$.
\end{proof}

\begin{corollary}\label{pure.thm}
Let $A\rightarrow B$ be a pure ring homomorphism, and $h:F\rightarrow G$ an 
$A$-linear map.
If $1_B\otimes h:B\otimes_AF\rightarrow B\otimes_AG$ is surjective,
then $h$ is surjective.
\qed
\end{corollary}

\begin{lemma}\label{pure2.thm}
Let $A\rightarrow B$ be a pure ring homomorphism, and $G$ an $A$-module.
If $B\otimes_A G$ is a finitely generated $B$-module, then
$G$ is finitely generated as an $A$-module.
\end{lemma}

\begin{proof}
Let $\theta_1,\ldots,\theta_r$ be generators of $B\otimes_A G$.
Then we can write $\theta_j=\sum_{i=1}^{s} b_{ij}\otimes g_{ij}$ for some
$s>0$, $b_{ij}\in B$, and $g_{ij}\in G$.
Let $F$ be the $A$-free module with the basis $\{f_{ij}\mid 1\leq i\leq s,\;
1\leq j\leq r\}$, and $h:F\rightarrow G$ be the $A$-linear map
given by $f_{ij}\mapsto g_{ij}$.
Then by construction, $1_B\otimes h$ is surjective.
By Corollary~\ref{pure.thm}, $h$ is surjective, and hence
$G$ is finitely generated.
\end{proof}

\begin{definition}[cf.~{\cite[(2.7)]{Hashimoto2}}]
An $\Bbb F_p$-algebra map $A\rightarrow B$ is said to be $e$-Dumitrescu 
if there exists some $e>0$ such that $\Phi_e(A,B)$ is 
$A$-pure (i.e., pure as an $A$-linear map).
\end{definition}

\begin{lemma}
Let $e,e'>0$.
If $A\rightarrow B$ is both $e$-Dumitrescu and $e'$-Dumitrescu, then 
it is $(e+e')$-Dumitrescu.
In particular, an $e$-Dumitrescu map is $er$-Dumitrescu map for $r>0$.
\end{lemma}

\begin{proof}
This follows from \cite[Lemma~4.1, {\bf 2}]{Hashimoto}.
\end{proof}

So a $1$-Dumitrescu map is Dumitrescu (that is, $e$-Dumitrescu for
all $e>0$), see \cite[Lemma~2.9]{Hashimoto2}.

\begin{lemma}\label{revision.thm}
Let $e>0$.
\begin{enumerate}
\item[\bf 1] {\rm\cite[Lemma~2.8]{Hashimoto2}},
\item[\bf 2] {\rm\cite[Lemma~2.12]{Hashimoto2}}, and
\item[\bf 3] {\rm\cite[Corollary~2.13]{Hashimoto2}}
\end{enumerate}
hold true when we replace all the
`Dumitrescu' therein by `$e$-Dumitrescu'.
\end{lemma}

The proof is straightforward, and is left to the reader.

\begin{theorem}\label{main.thm}
Let $f:A\rightarrow B$ and $g:B\rightarrow C$ be $\Bbb F_p$-algebra maps,
and $e>0$.
Assume that 
$g$ is $e$-Dumitrescu, and the image of the associated map
${}^ag:\Spec C\rightarrow \Spec B$ contains the set of maximal ideals
$\Max B$ of $B$.
If $gf$ is $F$-finite, and
$B$ and $C^{(e)}\otimes_{A^{(e)}}A$ are Noetherian,
then $f$ is $F$-finite.
\end{theorem}

\begin{proof}
Note that $\Phi_e(A,C):C^{(e)}\otimes_{A^{(e)}}A\rightarrow C$ is a finite map.
Note also that 
$C^{(e)}\otimes_{B^{(e)}}B$ is a $C^{(e)}\otimes_{A^{(e)}}A$-submodule 
of $C$ through $\Phi_e(B,C)$, since $\Phi_e(B,C)$ is $B$-pure and hence is
injective.
As $C^{(e)}\otimes_{A^{(e)}}A$ is Noetherian, 
$C^{(e)}\otimes_{B^{(e)}}B$, which is a submodule of the finite module $C$,
is a finite $C^{(e)}\otimes_{A^{(e)}}A$-module.
Since $g^{(e)}:B^{(e)}\rightarrow C^{(e)}$ is pure by
Lemma~\ref{revision.thm}, {\bf 2}, 
$B^{(e)}\otimes_{A^{(e)}}A\rightarrow C^{(e)}\otimes_{A^{(e)}}A$ is also pure.
Since
\[
C^{(e)}\otimes_{B^{(e)}}B\cong
(C^{(e)}\otimes_{A^{(e)}}A)\otimes_{B^{(e)}\otimes_{A^{(e)}}A}B
\]
is a finite $C^{(e)}\otimes_{A^{(e)}}A$-module, 
$B$ is a finite $B^{(e)}\otimes_{A^{(e)}}A$-module 
by Lemma~\ref{pure2.thm}.
\end{proof}

A homomorphism $f:A\rightarrow B$ between Noetherian rings is said to be
reduced if $f$ is flat with geometrically reduced fibers.

\begin{corollary}\label{main-cor.thm}
Let $g:B\rightarrow C$ be a faithfully flat reduced 
homomorphism between Noetherian
$\Bbb F_p$-algebras.
If $C$ is $F$-finite, then $B$ is $F$-finite.
\end{corollary}

\begin{proof}
By \cite[Theorem~3]{Dumitrescu}, $g$ is Dumitrescu.
As $g$ is faithfully flat, ${}^ag:\Spec C\rightarrow\Spec B$ is surjective.
Letting $A=\Bbb F_p$ and $f:A\rightarrow B$ be the unique map,
the assumptions of Theorem~\ref{main.thm} are satisfied, and
hence $f$ is $F$-finite.
That is, $B$ is $F$-finite.
\end{proof}

\begin{corollary}[Seydi \cite{Seydi}]
Let $(B,\fm)$ be a Nagata local ring with the $F$-finite residue field
$k=B/\fm$.
Then $B$ is $F$-finite.
In particular, $B$ is excellent, and is a homomorphic image of a regular
local ring.
\end{corollary}

\begin{proof}
Let $g:B\rightarrow C=\hat B$ be the completion of $B$.
Then $C$ is a complete local ring with the 
residue field $k$.
By Example~\ref{complete2.ex}, {\bf 3},
$C$ is $F$-finite.
As $g$ is reduced by \cite[(7.6.4), (7.7.2)]{EGA-IV}, 
$B$ is $F$-finite by Corollary~\ref{main-cor.thm}.

The last assertions follow from \cite[Theorem~2.5]{Kunz} and
\cite[Remark~13.6]{Gabber}.
\end{proof}

Even if $A\rightarrow B$ is a faithfully flat reduced homomorphism and
$B$ is excellent, $A$ need not be quasi-excellent.
There is a Nagata local ring $A$ which is not quasi-excellent 
\cite{Rotthaus2}, \cite{Nishimura2}, and its
completion $A\rightarrow \hat A=B$ is an example.

\end{document}